\documentclass [a4paper, 12pt, reqno]{amsart}
\usepackage [latin1]{inputenc}
\usepackage {a4}
\usepackage{amscd}
\usepackage{epsfig}
\usepackage{amssymb}
\usepackage{amsmath}
\usepackage{amsthm}
\usepackage[T1]{fontenc}
\usepackage{ae,aecompl}
\usepackage[arrow, matrix, curve]{xy}

\usepackage{color}

\usepackage{geometry}
\newcommand{\C} {\ensuremath{\mathbb{C}}}
\newcommand{\OO}{\mathcal{O}}

\newcommand{\dq}{\overline{\partial}}

\DeclareMathOperator{\Sing}{Sing}
\DeclareMathOperator{\supp}{supp}
\renewcommand{\Im}{\mbox{Im }}

\newtheorem {satz} {Satz} [section]
\newtheorem {lem} [satz] {Lemma}

\newtheorem {thm} [satz] {Theorem}

\title[$L^2$-Serre duality on singular complex spaces] 
{$L^2$-Serre duality on singular complex spaces and applications}

\author{J. Ruppenthal}

\address{Department of Mathematics, University of Wuppertal, Gau{\ss}str. 20, 42119 Wuppertal, Germany.}
\email{ruppenthal@uni-wuppertal.de}

\date{\today}



\begin{document}

\begin{abstract} 
In this survey, we explain a version of topological $L^2$-Serre duality for singular complex spaces with arbitrary singularities.
This duality can be used to deduce various $L^2$-vanishing theorems for the $\dq$-equation on singular spaces.
As one application, we prove Hartogs' extension theorem for $(n-1)$-complete spaces.
Another application is the characterization of rational singularities.
It is shown that complex spaces with rational singularities behave quite tame with respect to some $\dq$-equation
in the $L^2$-sense. More precisely: a singular point is rational if and only if the appropriate $L^2$-$\dq$-complex is exact in this point.
So, we obtain an $L^2$-$\dq$-resolution of the structure sheaf in rational singular points.
\end{abstract}

\maketitle

\section{Introduction}

Classical Serre duality, \cite{Serre}, can be formulated as follows: 
Let $X$ be a complex $n$-dimensional manifold, let $V\to X$ be a complex vector bundle, 
and let $\mathcal{E}^{0,q}(X,V)$ and $\mathcal{E}^{n,q}_c(X,V^*)$ be the spaces of global 
smooth $(0,q)$-form with values in $V$ and global smooth compactly supported $(n,q)$-forms 
with values in the dual bundle $V^*$, respectively. Then the following pairing is non-degenerate
\begin{equation}\label{eq:1}
H^q\big(\mathcal{E}^{0,\bullet}(X,V), \bar{\partial}\big) \times
H^{n-q}\big(\mathcal{E}_c^{n,\bullet}(X,V^*), \bar{\partial}\big) \to \mathbb{C}, \quad
([\varphi]_{\bar{\partial}}, [\psi]_{\bar{\partial}}) \mapsto \int_X \varphi\wedge\psi
\end{equation}
provided that $H^q\big(\mathcal{E}^{0,\bullet}(X,V), \bar{\partial}\big)$ and 
$H^{q+1}\big(\mathcal{E}^{0,\bullet}(X,V), \bar{\partial}\big)$ are 
Hausdorff topological vector spaces.

If $X$ is allowed to have singularities, then, traditionally, Serre duality takes a more algebraic and much less explicit form.
To explain that more precisely, let $\mathcal{F}:=\OO(F)$, $\mathcal{F}^*:=\OO(F^*)$
and let $\Omega^n_X$ denote the sheaf of holomorphic $n$-forms on $X$.
Then we can rephrase \eqref{eq:1} via the Dolbeault isomorphism algebraically:
There is a non-degenerate topological pairing
\begin{eqnarray}\label{eq:2}
H^q(X,\mathcal{F}) \times H^{n-q}(X,\mathcal{F}^*\otimes \Omega^n_X) \rightarrow \C,
\end{eqnarray}
realized by the cup-product, provided that $H^q(X,\mathcal{F})$ and $H^{q+1}(X,\mathcal{F})$ are Hausdorff.
In this formulation, Serre duality has been generalized to singular complex spaces, see,
e.g., Hartshorne \cite{Ha1}, \cite{Ha2} and Conrad \cite{C} for the algebraic setting and Ramis-Ruget \cite{RR}
and Andreotti-Kas \cite{AK} for the analytic setting.
In fact, if $X$ is of pure dimension $n$, paracompact and Cohen-Macaulay,
then there is again a non-degenerate topological pairing \eqref{eq:2} if we replace $\Omega^n_X$ 
by the Grothendieck dualizing sheaf $\omega_X$.
If $X$ is not Cohen-Macaulay, then $H^{n-q}(X,\mathcal{F}^*\otimes\Omega^n_X)$ has to be the replaced
by the cohomology of a certain complex of $\OO_X$-modules, called a dualizing complex.

\medskip
In this survey, we will explain how $L^2$-theory for the $\dq$-operator can be used to obtain
an $L^2$-version of Serre duality on singular spaces which has an analytic realization completely analogous to \eqref{eq:1}.
More precisely, we will show how \eqref{eq:1} generalizes to singular spaces
by replacing the Dolbeault cohomology groups of smooth $(0,q)$ and $(n,q)$-forms, respectively,
by $L^2$-Dolbeault cohomology groups.

\medskip
\section{$L^2$-theory for the $\dq$-operator on singular spaces}

The Cauchy-Riemann operator $\dq$ plays a fundamental role in Complex Analysis and Complex Geometry.
On complex manifolds, functions --  or more generally distributions -- are holomorphic if and only if they are in the kernel of the $\dq$-operator,
and the same holds in a certain sense on normal complex spaces.
For forms of arbitrary degree, the importance of the $\dq$-operator appears strikingly for example in the notion of $\dq$-cohomology
which can be used to represent the cohomology of complex manifolds by the Dolbeault isomorphism.

The $L^2$-theory for the $\dq$-operator is of particular importance in Complex Analysis and Geometry
and has become indispensable for the subject after the fundamental work of
H\"ormander on $L^2$-estimates and existence theorems for the $\dq$-operator \cite{Hoe1}
and the related work of Andreotti and Vesentini \cite{AnVe}.
Important applications of the $L^2$-theory are e.g. the Ohsawa-Takegoshi extension theorem \cite{OT},
Siu's analyticity of the level sets of Lelong numbers \cite{Siu0}
or the invariance of plurigenera \cite{Siu} -- just to name some.

\medskip
The first problem one has to face when studying the $\dq$-equation on singular spaces
is that it is not clear what kind of differential forms and operators one should consider.
Recently, there has been considerable progress by different approaches.

Andersson and Samuelsson developed in \cite {AS} Koppelman integral formulas for the $\dq$-equation on arbitrary singular complex spaces
which allow for a $\dq$-resolution of the structure sheaf in terms of certain fine sheaves of currents, called $\mathcal{A}$-sheaves.
These $\mathcal{A}$-sheaves are defined by an iterative procedure of repeated application of singular integral operators,
which makes them pretty abstract and hard to understand (and difficult to work with in concrete situations).

A second, more explicit approach is as follows: Consider
differential forms which are defined on the regular part of a singular variety
and which are square-integrable up to the singular set.
This setting seems to be very fruitful and has some history by now (see \cite{PS1}).\footnote{
The interest in this setting goes back to the invention of intersection (co-)homology by Goresky and MacPherson which has very tight connections
to the $L^2$-deRham cohomology of the regular part of a singular variety. 
We refer here to the solution of the Cheeger-Goresky-MacPherson conjecture \cite{CGM}
for varieties with isolated singularities by Ohsawa \cite{Oh1} (see \cite{PS1} for more details).}
Also in this direction, considerable progress has been made recently.
{\O}vrelid--Vassiliadou and the author obtained in \cite{OV3} and \cite{Rp8} a pretty complete description of the $L^2$-cohomology
of the $\dq$-operator (in the sense of distributions) at isolated singularities.

In this setting, we understand the class of objects with which we deal very well (just $L^2$-forms), 
but the disadvantage is a different one. Whereas the $\dq$-equation is locally solvable for closed $(0,q)$-forms
in the category of $\mathcal{A}$-sheaves by the Koppelman formulas in \cite{AS}, 
there are local obstructions to solving the $\dq$-equation in the $L^2$-sense at singular points (see e.g. \cite{FOV2},  \cite{OV3}, \cite{Rp8}).
So, there can be no $L^2$-$\dq$-resolution for the structure sheaf in general.

In this survey, we will see that the $\dq$-operator in the $L^2$-sense
behaves pretty well on spaces with canonical singularities which play a prominent role in the minimal model program.
The underlying idea is that canonical Gorenstein singularities are rational (see e.g. \cite{Kol}, Theorem 11.1),
i.e., we expect that the singularities do not contribute to the local cohomology.

Pursuing this idea, it turned out that there is a notion of $L^2$-$\dq$-cohomology for $(0,q)$-forms
which can be described completely in terms of a resolution of singularities (see \eqref{eq14} below).
A singular point is rational if and only if this certain $L^2$-$\dq$-complex is exact in this point.
If the underlying space has rational singularities, particularly on a Gorenstein space with canonical singularities,
then we obtain an $L^2$-$\dq$-resolution of the structure sheaf,
i.e., a resolution of the structure sheaf in terms of a well-known and easy to handle class of differential forms.
One of our main tools is a version of topological $L^2$-Serre duality for singular complex spaces with arbitrary singularities,
which seems to be useful in other contexts, too (Theorem \ref{thm:main1}).

\medskip
\section{Two $\dq$-complexes on singular complex spaces}

We need to specify what we mean by differential forms and the $\dq$-operator in the presence of singularities.
Let $X$ be a Hermitian complex space\footnote{A Hermitian complex space $(X,g)$ is a reduced complex space $X$ with a metric $g$ on the regular part
such that the following holds: If $x\in X$ is an arbitrary point there exists a neighborhood $U=U(x)$ and a
biholomorphic embedding of $U$ into a domain $G$ in $\C^N$ and an ordinary smooth Hermitian metric in $G$
whose restriction to $U$ is $g|_U$.}
of pure dimension $n$ and $F\rightarrow X$ a Hermitian holomorphic line bundle.
We denote by $\mathcal{L}^{p,q}(F)$ the sheaf of germs of $F$-valued $(p,q)$-forms on the regular part of $X$
which are square-integrable on $K^*=K\setminus\Sing X$ for any compact set $K$ in their domain of definition.\footnote{This is what we mean by
square-integrable up to the singular set.}
Note that $\mathcal{L}^{p,q}(F)$ becomes a Fr\'echet sheaf with the $L^{2,loc}$-topology on open subsets of $X$.

Due to the incompleteness of the metric on $X^*=X\setminus \Sing X$, there are different reasonable definitions
of the $\dq$-operator on $\mathcal{L}^{p,q}(F)$-forms. To be more precise, let $\dq_{cpt}$ be the $\dq$-operator
on smooth forms with support away from the singular set $\Sing X$.
Then $\dq_{cpt}$ can be considered as a densely defined operator $\mathcal{L}^{p,q}(F) \rightarrow \mathcal{L}^{p,q+1}(F)$.
One can now consider various closed extensions of this operator. The two most important are the maximal closed extension,
i.e., the $\dq$-operator in the sense of distributions which we denote by $\dq_w$,
and the minimal closed extension, i.e., the closure of the graph of $\dq_{cpt}$ which we denote by $\dq_s$.
Let $\mathcal{C}^{p,q}(F)$ be the domain of definition of $\dq_w$ which is a subsheaf of $\mathcal{L}^{p,q}(F)$,
and $\mathcal{F}^{p,q}(F)$ the domain of definition of $\dq_s$ which in turn is a subsheaf of $\mathcal{C}^{p,q}(F)$.
We obtain complexes of fine sheaves
\begin{eqnarray}\label{eq:intro1}
\mathcal{C}^{p,0}(F) \overset{\dq_w}{\longrightarrow} \mathcal{C}^{p,1}(F) \overset{\dq_w}{\longrightarrow} 
\mathcal{C}^{p,2}(F) \overset{\dq_w}{\longrightarrow} ...
\end{eqnarray}
and
\begin{eqnarray}\label{eq:intro2}
\mathcal{F}^{p,0}(F) \overset{\dq_s}{\longrightarrow} \mathcal{F}^{p,1}(F) \overset{\dq_s}{\longrightarrow} \mathcal{F}^{p,2}(F) \overset{\dq_s}{\longrightarrow} ...
\end{eqnarray}

We refer to \cite{Rp11} for more details, but let us explain the $\dq_s$-operator more precisely for convenience of the reader.
Let $f$ be a germ in $\mathcal{C}^{p,q}(F)$, i.e., an $F$-valued $(p,q)$-form on an open set $U$ in $X$ (living on the regular part of $U$)
which is $L^2$ on compact subsets of $U$ and such that the $\dq$ in the sense of distributions, $\dq_w f$, is in the same class of forms.
Then $f$ is in the domain of the $\dq_s$-operator (and we set $\dq_s f=\dq_w f$) if there exists a sequence of forms $\{f_j\}_j \subset \mathcal{C}^{p,q}(U,F)$
with support away from the singular set, $\supp f_j\cap \Sing X=\emptyset$, such that
\begin{eqnarray*}
f_j \rightarrow f  &\mbox{ in }& \mathcal{L}^{p,q}(U,F),\\
\dq_w f_j \rightarrow \dq_w f  &\mbox{ in }& \mathcal{L}^{p,q+1}(U,F).
\end{eqnarray*}
This means that the $\dq_s$-operator comes with a certain Dirichlet boundary at the singular set of $X$,
which can also be interpreted as a growth condition. We have e.g. the following:

\begin{lem}[\cite{Rp11}]\label{lemma}
Bounded forms in the domain of $\dq_w$ are in the domain of $\dq_s$.
\end{lem}

If $F$ is just the trivial line bundle, then $\mathcal{K}_X:=\ker\dq_w \subset \mathcal{C}^{n,0}$ is the canonical sheaf of Grauert--Riemenschneider (see \cite{GrRie})
and $\mathcal{K}_X^s:=\ker\dq_s \subset \mathcal{F}^{n,0}$ is the sheaf of holomorphic $n$-forms with Dirichlet boundary condition
that was introduced in \cite{Rp8}.
We will see below that $\widehat{\OO}_X=\ker\dq_s \subset \mathcal{F}^{0,0}$ for the sheaf of weakly holomorphic functions $\widehat{\OO}_X$.

It is clear that \eqref{eq:intro1} and \eqref{eq:intro2} are exact in regular points of $X$.
Exactness in singular points is equivalent to the difficult problem of solving $\dq$-equations locally in the $L^2$-sense at singularities,
which is not possible in general (see e.g. \cite{FOV2}, \cite{OV1}, \cite{OV3}, \cite{Rp1}, \cite{Rp7}, \cite{Rp8}).
However, it is known that \eqref{eq:intro1} is exact for $p=n$ (see \cite{PS1}),
and that \eqref{eq:intro2} is exact for $p=n$ if $X$ has only isolated singularities
(see \cite{Rp8}).
In these cases, the complexes \eqref{eq:intro1} and \eqref{eq:intro2} are fine resolutions of the canonical sheaves $\mathcal{K}_X$ and $\mathcal{K}_X^s$,
respectively.

For an open set $\Omega\subset X$, we denote by $H^{p,q}_{w,loc}(\Omega,F)$ the cohomology of the complex \eqref{eq:intro1},
and by $H^{p,q}_{w,cpt}(\Omega,F)$ the cohomology of \eqref{eq:intro1} with compact support. Analogously, let $H^{p,q}_{s,loc}(\Omega,F)$
and $H^{p,q}_{s,cpt}(\Omega,F)$ be the cohomology groups of \eqref{eq:intro2}.
These $L^2$-cohomology groups inherit the structure of topological vector spaces,
which are locally convex Hausdorff spaces if the corresponding $\dq$-operators have closed range.\footnote{
Note that different Hermitian metrics lead to $\dq$-complexes which are equivalent on relatively compact subsets.
So, one can put any Hermitian metric on $X$ in many of the results below.}

\medskip
\section{$L^2$-Serre duality}

We can now formulate the $L^2$-version of \eqref{eq:1} for singular complex spaces:

\begin{thm}[{\bf Serre duality \cite{Rp11}}]\label{thm:main1}
Let $X$ be a Hermitian complex space of pure dimension $n$, $F\rightarrow X$ a Hermitian holomorphic line bundle, and let $0\leq p,q \leq n$.
If $H^{p,q}_{w,loc}(\Omega,F)$ and $H^{p,q+1}_{w,loc}(\Omega,F)$ are Hausdorff, then the mapping
\begin{eqnarray*}
\mathcal{L}^{p,q}(\Omega,F) \times \mathcal{L}^{n-p,n-q}_{cpt}(\Omega,F^*) \rightarrow \C\ \ ,\ (\eta,\omega) \mapsto \int_{\Omega^*} \eta\wedge\omega,
\end{eqnarray*}
induces a non-degenerate pairing of topological vector spaces
\begin{eqnarray*}
H^{p,q}_{w,loc}(\Omega,F) \times H^{n-p,n-q}_{s,cpt}(\Omega,F^*) \rightarrow \C
\end{eqnarray*}
such that $H^{n-p,n-q}_{s,cpt}(\Omega,F^*)$ is the topological dual of $H^{p,q}_{w,loc}(\Omega,F)$ and vice versa.

The same statement holds with the indices $\{s, w\}$ in place of $\{w, s\}$. Then there is a non-degenrate pairing
\begin{eqnarray*}
H^{p,q}_{s,loc}(\Omega,F) \times H^{n-p,n-q}_{w,cpt}(\Omega,F^*) \rightarrow \C.
\end{eqnarray*}
\end{thm}

\smallskip
If the topological vector spaces $H^{p,q}_{w/s,loc}(\Omega,F)$, $H^{p,q+1}_{w/s,loc}(\Omega,F)$ are non-Hausdorff,
then the statement of Theorem \ref{thm:main1} holds at least for the
separated cohomology groups $\overline{H}_{w/s} = \ker \dq_{w/s} / \overline{\Im \dq_{w/s}}$.\footnote{The notation $w/s$ refers
either to the index $w$ or the index $s$ in the whole statement.}
The two main difficulties in the proof of Theorem \ref{thm:main1} are as follows. First, the $\dq$-operators under consideration are just
closed densely defined operators in the Fr\'echet spaces $\mathcal{L}^{p,q}(\Omega,F)$ and the $(LF)$-spaces $\mathcal{L}^{n-p,n-q}_{cpt}(\Omega,F^*)$.
Second, we have to show that the operators $\dq_w$ and $\dq_s$ are topologically dual, even at singularities.
Note that $H^{p,q}_{w/s,loc}(\Omega,F)$ is Hausdorff if and only if $\dq_{w/s}$ has closed range in $\mathcal{L}^{p,q}(\Omega,F)$,
and to decide whether this is the case is usually as difficult as solving the corresponding $\dq$-equation.
Using local $L^2$-$\dq$-solution results for singular spaces, one can show at least:

\begin{thm}[{\bf \cite{Rp11}}]\label{thm:main2}
Let $X$ be a Hermitian complex space of pure dimension $n$, $F\rightarrow X$ a Hermitian holomorphic line bundle, and let $0\leq p,q \leq n$.
Let $\Omega \subset X$ be a holomorphically convex open subset. Then the topological vector spaces
$$H^{n,q}_{w,loc}(\Omega,F)\ \ ,\ \ H^{n,q}_{w,cpt}(\Omega,F)\ \ ,\ \ H^{0,n-q}_{s,cpt}(\Omega,F^*)\ \ ,\ \ H^{0,n-q}_{s,loc}(\Omega,F^*)$$
are Hausdorff for all $0\leq q\leq n$.
\end{thm}

A main point in the proof of Theorem \ref{thm:main2} is to show that the canonical Fr\'echet sheaf structure
of compact convergence on the coherent analytic canonical sheaf $\mathcal{K}_X$
coincides with the Fr\'echet sheaf structure of $L^2$-convergence on compact subsets.
This allows then to show also the topological equivalence of \v{C}ech cohomology and $L^2$-cohomology.
If $X$ has only isolated singularities, then the Hausdorff property is known also for some cohomology spaces 
of different degree (see \cite{Rp11}).

As a direct application of Serre duality, Theorem \ref{thm:main1}, one can deduce:

\begin{thm}\label{thm:main3}
Let $X$ be a Hermitian complex space of pure dimension $n$, $F\rightarrow X$ a Hermitian holomorphic line bundle
and $\Omega\subset X$ a cohomologically $q$-complete open subset, $q\geq 1$. Then
\begin{eqnarray*}
H^{n,r}_{w,loc}(\Omega,F) = H^{0,n-r}_{s,cpt}(\Omega,F^*)=0\ \ \ \mbox{ for all }\  r\geq q.
\end{eqnarray*}
\end{thm}

Note that $\Omega$ is cohomologically $q$-complete if it is $q$-complete by the Andreotti-Grauert vanishing theorem \cite{AG}.
So, Theorem \ref{thm:main3} allows to solve the $\dq_s$-equation with compact support for $(0,n-q)$-forms on $q$-complete spaces,
which is of particular interest for $1$-complete spaces, i.e., Stein spaces.

\medskip
\section{Hartogs' extension theorem}

Let us mention some applications. As an interesting consequence of Theorem \ref{thm:main3}, we obtain Hartogs'
extension theorem in its most general form. This version of the Hartogs' extension was first obtained by Merker-Porten \cite{MePo2}
and shortly thereafter also by Coltoiu-Ruppenthal \cite{CR}.
Merker and Porten gave an involved geometrical proof by using a finite number of
parameterized families of holomorphic discs and Morse-theoretical tools for the global topological control 
of monodromy, but no $\dq$-theory.
Shortly after that, Coltoiu and Ruppenthal were able to give a short $\dq$-theoretical proof by the Ehrenpreis-$\dq$-technique (cf. \cite{CR}).
This approach involves Hironaka's resolution of singularities which may be considered a very deep theorem.
In the present survey, we give a very short proof of the extension theorem by the Ehrenpreis-$\dq$-technique
without needing a resolution of singularities. We just use the vanishing result $H^{0,1}_{s,cpt}(X)=0$

\begin{thm}\label{thm:hartogs}
Let $X$ be a connected normal complex space of dimension $n\geq 2$ which is cohomologically $(n-1)$-complete.
Furthermore, let $D$ be a domain in $X$ and $K\subset D$ a compact subset such that $D\setminus K$ is connected.
Then each holomorphic function $f\in \mathcal{O}(D\setminus K)$ has a unique holomorphic extension to the whole set $D$.
\end{thm}

\begin{proof}
Let $f\in\OO(D\setminus K)$. Choose a cut-off function $\chi\in C^\infty_{cpt}(D)$ such that $\chi$ is identically $1$
in a neighborhood of $K$. Then $g:=(1-\chi)f$ is an extension of $f$, but unfortunately not holomorphic.
However, we can fix it by the $\dq$-strategy of Ehrenpreis. By Lemma \ref{lemma}, $g$ is in the domain of $\dq_s$
and $H^{0,1}_{s,cpt}(X)=0$ by Theorem \ref{thm:main3}. So, there exists a solution $h$ to the $\dq_s$-equation with compact support
$\dq_s h=\dq_s g$ and $F:=g - h$ is the desired extension  of $f$ to the whole of $D$. That can be seen by use of the identity theorem
and the fact that $X$ cannot be compact (because Theorem \ref{thm:main3} implies also that $H^{0,0}_{s,cpt}(X)=0$).
\end{proof}

\medskip
\section{Rational singularities}

Another, very interesting application of $L^2$-Serre duality is the following characterization of rational singularities.
Let $\pi: M \rightarrow X$ be a resolution of singularities and $\Omega\subset \subset X$ holomorphically convex.
Give $M$ any Hermitian metric. Then pullback of $L^2$-$(n,q)$-forms under $\pi$ induces an isomorphism
\begin{equation}\label{eq13}
\pi^*: H^{n,q}_{w,cpt}\big(\Omega\big) \overset{\cong}{\longrightarrow} 
H^{n,q}_{w,cpt}\big(\pi^{-1}(\Omega)\big) \cong H^q_{cpt}\big(\pi^{-1}(\Omega),\mathcal{K}_M\big)
\end{equation}
for all $0\leq q\leq n$ by use of Pardon--Stern \cite{PS1} and the Takegoshi vanishing theorem \cite{Ta} (see \cite{Rp11} for more details).
Now we can use the $L^2$-Serre duality, Theorem \ref{thm:main1}, and classical Serre duality on the smooth manifold $\pi^{-1}(\Omega)$
to deduce that push-forward of forms under $\pi$ induces another isomorphism
\begin{equation}\label{eq14}
\pi_*: H^{n-q}\big(\pi^{-1}(\Omega),\OO_M\big) \overset{\cong}{\longrightarrow}
H^{0,n-q}_{s,loc} \big(\Omega\big)
\end{equation}
for all $0\leq q\leq n$ (see \cite{Rp11}, Theorem 1.1). This shows that the obstructions to solving the $\dq_s$-equation locally for $(0,q)$-forms
can be expressed in terms of a resolution of singularities.
For the cohomology sheaves of the complex $(\mathcal{F}^{0,\bullet},\dq_s)$, we see that
$$\left(\mathcal{H}^q\big(\mathcal{F}^{0,\bullet},\dq_s\big)\right)_x \cong \left(R^q \pi_* \OO_M\right)_x$$
in any point $x\in X$ for all $q\geq 0$.
It follows that the functions in the kernel of $\dq_s$ are precisely the weakly holomorphic functions,
and for $p=0$ the complex \eqref{eq:intro2} is exact in a point $x\in X$ exactly if $x$ is a rational point: 

\begin{thm}[{\bf \cite{Rp11}, Theorem 1.3}]\label{thm:main4}
Let $X$ be a Hermitian complex space. Then the $L^2$-$\dq$-complex
\begin{eqnarray}\label{eq:exactnessM01}
0\rightarrow \OO_X \longrightarrow \mathcal{F}^{0,0} \overset{\dq_s}{\longrightarrow}
\mathcal{F}^{0,1} \overset{\dq_s}{\longrightarrow} \mathcal{F}^{0,2} 
\overset{\dq_s}{\longrightarrow} \mathcal{F}^{0,3} \overset{\dq_s}{\longrightarrow} ...
\end{eqnarray}
is exact in a point $x\in X$ if and only if $x$ is a rational point.

Hence, if $X$ has only rational singularities, then \eqref{eq:exactnessM01}
is a fine resolution of the structure sheaf $\OO_X$. 
\end{thm}

Recall that a point $x\in X$ is rational if it is a normal point and $\big(R^q \pi_* \OO_M\big)_x=0$ for all $q\geq 1$.
If $X$ has only rational singularities, then Theorem \ref{thm:main4} yields immediately further finiteness and vanishing results,
e.g. if $X$ is $q$-convex or $q$-complete.

\medskip
Let us point out also the following interesting fact.
Let $X$ be a Gorenstein space with canonical singularities.
By exactness of \eqref{eq:exactnessM01} and exactness of \eqref{eq:intro1} for $p=n$,
the non-degenerate $L^2$-Serre duality pairing
$$H^{0,q}_{s,loc}(\Omega) \times H^{n,n-q}_{w,cpt}(\Omega) \rightarrow \C\ ,\ ([\eta],[\omega]) \mapsto \int_{\Omega^*} \eta\wedge \omega,$$
is for $0\leq q\leq n$ then an explicit realization of Grothendieck duality after Ramis-Ruget \cite{RR},
$$\left( H^q(\Omega,\OO_X)\right)^* \cong H^{n-q}_{cpt}(\Omega,\omega_X),$$
given the cohomology groups under consideration are Hausdorff.
Here, $\omega_X$ denotes the Grothendieck dualizing sheaf which coincides with the Grauert-Riemenschneider canonical sheaf $\mathcal{K}_X$
as $X$ has canonical Gorenstein singularities.

\medskip
\section{$\mathcal{A}$-sheaf duality}

We conclude by mentioning another approach to analytic Serre duality on singular complex spaces
which is based on the so-called $\mathcal{A}_{0,q}$-sheaves introduced
by Andersson and Samuelsson in \cite{AS}. These are certain sheaves of $(0,q)$-currents 
on singular complex spaces which are smooth on the regular part of the variety
and such that the $\dq$-complex
\begin{equation}\label{eq2}
0 \rightarrow \mathcal{O}_X \hookrightarrow \mathcal{A}_{0,0} \overset{\dq}{\longrightarrow} 
\mathcal{A}_{0,1} \overset{\dq}{\longrightarrow} \mathcal{A}_{0,2} \longrightarrow ...
\end{equation}
is a fine resolution of the structure sheaf. The $\mathcal{A}$-sheaves are defined via Koppelman integral formulas
on singular complex spaces.

Analogously, in \cite{RSW2},
we introduced a $\dq$-complex of fine sheaves of $(n,q)$-currents (smooth on the regular part of the variety)
\begin{equation}\label{eq3}
0 \rightarrow \omega_X \hookrightarrow \mathcal{A}_{n,0} \overset{\dq}{\longrightarrow} 
\mathcal{A}_{n,1} \overset{\dq}{\longrightarrow} \mathcal{A}_{n,2} \longrightarrow ...
\end{equation}
where $X$ is of pure dimension $n$ and $\omega_X$ denotes the Grothendieck dualizing sheaf.
The complex \eqref{eq3} is exact only under some additional assumptions, e.g. if $X$ is Cohen-Macaulay.
We call $(\mathcal{A}_{n,\bullet},\dq)$ a dualizing Dolbeault complex for $\OO_X$ because
we obtain in \cite{RSW2} a non-degenerate topological pairing
\begin{equation}\label{eq4}
H^q\big(\mathcal{A}_{0,\bullet}(X), \bar{\partial}\big) \times
H^{n-q}_{cpt}\big(\mathcal{A}_{n,\bullet}(X), \bar{\partial}\big) \to \mathbb{C}, \quad
([\varphi]_{\bar{\partial}}, [\psi]_{\bar{\partial}}) \mapsto \int_X \varphi\wedge\psi,
\end{equation}
provided that $H^q(X,\OO_X) \cong H^q\big(\mathcal{A}_{0,\bullet}(X), \bar{\partial}\big)$ and 
$H^{q+1}(X,\OO_X) \cong H^{q+1}\big(\mathcal{A}_{0,\bullet}(X), \bar{\partial}\big)$ are 
Hausdorff topological spaces.

\bigskip
{\bf Acknowledgments.}
The author was supported by the Deutsche Forschungsgemeinschaft (DFG, German Research Foundation), 
grant RU 1474/2 within DFG's Emmy Noether Programme.


\bigskip
\begin{thebibliography}{99999}


\bibitem[AS]{AS}{\sc M.\ Andersson, H.\ Samuelsson},
A Dolbeault-Grothendieck lemma on complex spaces via Koppelman formulas,
{\em Invent. Math.} {\bf 190} (2012), no. 2, 261--297.


\bibitem[AG]{AG}{\sc A.\ Andreotti, H.\ Grauert}, Th\'eor\`eme de finitude pour la cohomologie des espaces complexes,
{\em Bull. Soc. Math. France} {\bf 90} (1962), 193--259.


\bibitem[AK]{AK}{\sc A.\ Andreotti, A.\ Kas},
Duality on complex spaces, {\em Ann. Scuola Norm. Sup. Pisa} (3) {\bf 27} (1973), 187--263.


\bibitem[AV]{AnVe} {\sc A.\ Andreotti, E.\ Vesentini},
Carleman estimates for the Laplace Beltrami equation on complex manifolds,
{\em Publ. Math. Inst. Hautes Etudes Sci.} {\bf 25} (1965), 81--130.


\bibitem[CGM]{CGM}{\sc J.\ Cheeger, M.\ Goresky, R.\ MacPherson},
{\em $L^2$-cohomology and intersection homology of singular algebraic varieties},
Ann. of Math. Stud., No. 102, Princeton University Press, Princeton, 1982, 303--340.


\bibitem[CR]{CR} {\sc M.\ Coltoiu, J.\ Ruppenthal}, On Hartogs' extension theorem on $(n-1)$-complete spaces,
{\em J. reine angew. Math} {\bf 637} (2009), 41--47.


\bibitem[C]{C}{\sc B.\ Conrad},
{\em Grothendieck duality and base change}, Lecture Notes in Mathematics, 1750, Springer-Verlag, Berlin, 2000.


\bibitem[FOV]{FOV2} {\sc J.\ E.\ Forn{\ae}ss, N.\ {\O}vrelid, S.\ Vassiliadou},
Local $L^2$ results for $\dq$: the isolated singularities case,
{\em Internat. J. Math.} {\bf 16} (2005), {\em no. 4}, 387--418.


\bibitem[GR]{GrRie} {\sc H.\ Grauert, O.\ Riemenschneider},
Verschwindungss\"atze f\"ur analytische Kohomologiegruppen auf komplexen R\"aumen,
{\em Invent. Math.} {\bf 11} (1970), 263--292.


\bibitem[H1]{Ha1}{\sc R.\ Hartshorne},
{\em Algebraic geometry}, Graduate Texts in Mathematics, No. 52, Springer-Verlag, New York-Heidelberg, 1977.


\bibitem[H2]{Ha2}{\sc R.\ Hartshorne},
{\em Residues and duality}, Lecture Notes in Mathematics, No. 20, Springer-Verlag, Berlin-New York, 1966.


\bibitem[H3]{Hoe1}{\sc L.\ H\"ormander}, $L^2$-estimates and existence theorems for the $\dq$-operator,
{\em Acta Math.} {\bf 113} (1965), 89--152.


\bibitem[K]{Kol}{\sc J.\ Koll\'ar}, Singularities of pairs. {\em Algebraic geometry -- Santa Cruz 1995}, 221--287,
Proc. Sympos. Pure Math. {\bf 62}, Part 1, Amer. Math. Soc., Providence, RI, 1997.


\bibitem[MP]{MePo2}{\sc J.\ Merker, E.\ Porten}, The Hartogs' extension theorem on $(n-1)$-complete complex spaces,
{\em J. reine u. angw. Math.} {\bf 637} (2009), 23--39.


\bibitem[O]{Oh1}{\sc T.\ Ohsawa},
Cheeger-Goresky-MacPherson's conjecture for the varieties with isolated singularities, 
{\em Math. Z.} {\bf 206} (1991), 219--224.


\bibitem[OT]{OT} {\sc T.\ Ohsawa, K.\ Takegoshi},
On the extension of $L^2$-holomorphic functions,
{\em Math. Z.} {\bf 195} (1987), no. 2, 197--204.


\bibitem[OV1]{OV1} {\sc N.\ {\O}vrelid, S.\ Vassiliadou}, Solving $\dq$ on product singularities,
{\em Complex Var. Ellipitic Equ.} {\bf 51} (2006), {\em no. 3}, 225--237.


\bibitem[OV2]{OV3}{\sc N.\ {\O}vrelid, S.\ Vassiliadou}, $L^2$-$\dq$-cohomology groups of some singular complex spaces,
{\em Invent. Math.} {\bf 192} (2013), no.2, 413--458.


\bibitem[PS]{PS1}{\sc W.\ Pardon, M.\ Stern},
$L^2$-$\dq$-cohomology of complex projective varieties, {\em J. Amer. Math. Soc.} {\bf 4} (1991), no. 3, 603--621.


\bibitem[RR]{RR}{\sc J.-P.\ Ramis, G.\ Ruget}, Complexe dualisant et th\'eor\`eme de dualit\'e en g\'eom\'etrie analytique complexe,
{\em Inst. Hautes \'Etudes Sci. Publ. Math.} {\bf 38} (1970), 77--91.


\bibitem[R1]{Rp1}{\sc J.\ Ruppenthal}, About the $\dq$-equation at isolated singularities with regular exceptional set,
{\em Internat. J. Math.} {\bf 20} (2009), no. 4, 459--489.


\bibitem[R2]{Rp7}{\sc J.\ Ruppenthal}, The $\dq$-equation on homogeneous varieties with an isolated singularity,
{\em Math. Z.} {\bf 263} (2009), 447--472.


\bibitem[R3]{Rp8} {\sc J.\ Ruppenthal}, $L^2$-theory for the $\dq$-operator on compact complex spaces,
{\sf arXiv:1004.0396}, to appear in {\em Duke Math. J.}


\bibitem[R4]{Rp11}{\sc J.\ Ruppenthal},
Serre duality and $L^2$-vanishing theorems on singular spaces, Preprint 2014,
{\sf arXiv:1401.4563}, submitted.


\bibitem[RSW]{RSW2}{\sc J.\ Ruppenthal, H.\ Samuelsson Kalm, E. Wulcan},
Explicit Serre duality on complex spaces, Preprint 2014,
{\sf arXiv:1401.8093}, submitted.


\bibitem[S1]{Serre}{\sc J.-P.\ Serre}, Un th\'eor\`eme de dualit\'e. Comm. Math. Helv. {\bf 29} (1955), 9--26.


\bibitem[S2]{Siu0}{\sc Y.-T.\ Siu},
Analyticity of sets associated to Lelong numbers and the extension of closed positive currents,
{\em Invent. Math.} {\bf 27} (1974), 53--156.

\bibitem[S3]{Siu}{\sc Y.-T.\ Siu},
Invariance of Plurigenera, {\em Invent. Math.} {\bf 134} (1998), no. 3, 661--673.


\bibitem[T]{Ta} {\sc K.\ Takegoshi}, Relative vanishing theorems in analytic spaces,
{\em Duke Math. J.} {\bf 51} (1985), no. 1, 273--279.


\end{thebibliography}
\end{document}